\documentclass{amsart}
\usepackage{amsthm}
\usepackage{amsfonts}
\usepackage{amsmath}
\usepackage{caption}
\usepackage{graphicx}
\usepackage{comment}
\usepackage{amscd}
\usepackage{tikz-cd}

 \newcommand{\GG}{\mathbb{G}}
 \newcommand{\PP}{\mathbb{P}}
 \newcommand{\KK}{\mathbb{K}}
\begin{document}

\title{Commutative actions on smooth projective quadrics}
\author{Viktoriia Borovik}
\address{Lomonosov Moscow State University, Faculty of Mechanics and Mathematics, Department of Higher Algebra, Leninskie Gory 1, Moscow, 119991 Russia; \linebreak and \linebreak
National Research University Higher School of Economics, Faculty of Computer Science, Pokrovsky Boulevard 11, Moscow, 109028, Russia}
\email{vborovik@hse.ru}

\author{Sergey Gaifullin}
\address{Lomonosov Moscow State University, Faculty of Mechanics and Mathematics, Department of Higher Algebra, Leninskie Gory 1, Moscow, 119991 Russia; \linebreak and \linebreak
National Research University Higher School of Economics, Faculty of Computer Science, Pokrovsky Boulevard 11, Moscow, 109028, Russia}
\email{sgayf@yandex.ru}

\author{Anton Trushin}
\address{Lomonosov Moscow State University, Faculty of Mechanics and Mathematics, Department of Higher Algebra, Leninskie Gory 1, Moscow, 119991 Russia}
\email{ornkano@mail.ru}


\thanks{The second author was supported by RSF grant 19-11-00172.}
\date{\today}
\maketitle 

\newtheorem{theorem}{Theorem}
\newtheorem{lemma}{Lemma}
\newtheorem{proposition}{Proposition}
\newtheorem{cor}{Corollary}
\theoremstyle{remark}
\newtheorem{remark}{Remark}
\theoremstyle{definition}
\newtheorem{definition}{Definition}

\sloppy
\textwidth=16.3cm
\oddsidemargin=0cm
\topmargin=0cm
\headheight=0cm
\headsep=1cm
\textheight=23.5cm
\evensidemargin=0cm

\begin{abstract}

By a commutative action on a smooth quadric $Q_{n}$ in $\mathbb{P}^{n+1}$ we mean an effective action of a commutative connected algebraic group on $Q_{n}$ with an open orbit. 
We show that for $n\geq 3$ all commutative actions on $Q_n$ are additive actions described by Sharoiko in 2009. So there is a unique commutative action on $Q_n$ up to equivalence. For $n=2$ there are three commutative actions on $Q_2$ up to equivalence, for $n=1$ there are two commutative actions on $Q_1$ up to equivalence.
\end{abstract}

\section{Introduction}

Let $\mathbb{K}$ be an algebraically closed field of characteristic zero. Let $\mathbb{G}_{a}$ and $\mathbb{G}_{m}$ be the additive group $(\mathbb{K}, +)$ and the multiplicative group $(\mathbb{K}^{\times}, \cdot)$, respectively. By a {\it commutative action} on an $n$-dimensional projective algebraic variety $X$ we mean an effective regular action of a commutative linear algebraic group $G$ on $X$ with an open orbit. It is well known that $G$ splits into direct product of multiplicative and additive groups, i.e. $G \simeq \mathbb{G}_{a}^{l} \times \mathbb{G}_{m}^{r}$, where $n=l+r=\dim G$. If $l=0$ and $r=n$ we obtain a well-known theory of toric varieties. In the opposite case $r=0$ and $l=n$ we have an action of commutative unipotent group $\GG_a^n$. Such actions are called {\it additive actions}.
We call two actions $\alpha\colon~G\times~X\rightarrow X$ and $\beta\colon G\times~X\rightarrow~X$ {\it equivalent} if there are an automorphism $\varphi\colon G\rightarrow~G$ and an automorphism $\psi\colon X\rightarrow~X$ such that $\alpha(\varphi(g),\psi(x))=\psi(\beta(g,x)).$ If we are given a projective variety $X$ it is a natural problem to describe all commutative actions on~$X$ up to equivalence. 

This problem is completely solved when $X$ is a projective space.  The study of commutative actions on projective spaces was initiated by Knop and Lange. They showed in \cite{KL} that equivalence classes of commutative actions on the projective space $\mathbb{P}^{n}$ are in bijection with isomorphism classes of ${(n+1)}$-dimensional commutative associative algebras with unit. The  particular case of additive actions on the projective space  $\PP^n$ became popular after the paper \cite{HTsch} by Hassett and Tschinkel. They prove that additive actions correspond to local commutative associative algebras with unit. Moreover, there is more general correspondence. Let us say that an action of a group on $\mathbb{P}^n$ is {\it cyclic} if there exists an orbit that is not contained in any proper subspace.  There is a bijection between equivalence classes of effective cyclic $\GG_a^n$-actions on $\PP^m$ and isomorphism classes of pairs $(A,U)$, where~$A$ is a local commutative associative $(m+1)$-dimensional algebra with the maximal ideal~$\mathfrak{m}$ and $U$ is an $n$-dimensional subspace of $\mathfrak{m}$ that generates $A$ as an algebra with unit. These correspondences inspired recent works studying additive actions on projective varieties, see, e.g. \cite{A, AR, D, F, FH, L}.

The next natural particular case of the problem above is the case of hypersurfaces $X_n\subseteq \PP^{n+1}$. In 2009, Sharoiko \cite{Sh} used the Hassett-Tschinkel correspondence to describe all additive actions on smooth projective quadrics $Q_n\subseteq \PP^{n+1}$, see also~\cite{ASh}. It was shown that for each $n$ there is a unique additive action on~$Q_n$ up to equivalence.  The main ideas of  \cite{Sh} is as follows. Each additive action on $Q_n$ can be lifted to an effective $\GG_a^n$-action on $\PP^{n+1}$. Such an action corresponds to a local commutative associative $(n+2)$-dimensional algebra $A$ with an $n$-dimensional subspace $U$ in $\mathfrak{m}$ generating~$A$. The condition that $Q_n$ is invariant under the action allows to define a scalar product $F$ on $A$ having some relations with multiplication on~$A$. It turns out that for each $n$ the triple $(A,U,F)$ is unique up to isomorphism.

In~\cite{AP} Arzhantsev and Popovskiy describe all additive actions on quadrics of corank 1. In \cite{B} Bazhov proves some generalizations for cubics. In this paper we give another generalization of the result of~\cite{Sh}. We describe all commutative actions on smooth projective quadrics. 

\begin{theorem}
There are only three equivalence classes of non-additive commutative action on smooth projective quadric:
\begin{itemize}
    \item $\mathbb{G}_m$-action on $Q_1$;
    \item $\mathbb{G}_m^2$-action on $Q_2$;
    \item $\mathbb{G}_a\times\mathbb{b}G_m$-action on $Q_2$.
\end{itemize}In particular, for all $n\geq 3$ there is no commutative action on an $n$-dimensional smooth projective quadric $Q_n$ that are not additive. 
\end{theorem}

To obtain this result we generalize the Hassett-Tschinkel  technique and give a correspondence between  equivalence classes of effective cyclic $\GG_a^l\times \GG_m^r$-actions on~$\PP^m$ and isomorphism classes of pairs $(A,U)$, where $A$ is an $(m+1)$-dimensional algebra splitting into direct sum of $r+1$ local ones, $U$ is an $n=(r+l)$-dimensional subspace generating $A$ as an algebra with unit and the dimension of the intersection of~$U$ with the radical~$R$ of~$A$ has dimension $l$, see Theorem~\ref{t1}.  
After this we proceed analogically to~\cite{Sh}. We lift a commutative action on $Q_n$ to a $\GG_a^l\times\GG_m^r$-action on~$\PP^{n+1}$. Then we define a scalar product on the corresponding $(n+2)$-dimensional algebra $A$. Finally, we describe all algebras with generating subspace and scalar product corresponding to commutative actions on $Q_n$.

The authors are grateful to Ivan Arzhantsev for useful discussions and comments.

\section{Preliminaries}

This section presents a collection of definitions and results from the theory of actions of commutative linear algebraic groups on projective varieties.

Let $X$ be an irreducible projective algebraic variety of dimension $n$ and $G~\simeq~\GG_a^l \times \GG_m^r$ be an $n$-dimensional commutative linear algebraic group.

\begin{definition}
A  \textit{commutative action} on $X$ is an effective action $G \times X \to X$ with an open orbit. Two commutative actions $\alpha$ and $\beta$ are said to be \textit{equivalent} if there are an automorphism $\varphi\colon G\rightarrow~G$ and an automorphism $\psi\colon X\rightarrow~X$ such that $\alpha(\varphi(g),\psi(x))=\psi(\beta(g,x)).$
\end{definition}

When $G \simeq \GG_a^n$ we obtain an \textit{additive action} on a projective variety. 

\begin{proposition} \label{p1}
{\rm \cite[Proposition 2.15]{HTsch}} There is a one-to-one correspondence
between

\begin{enumerate}
\item[(1)] equivalence classes of additive actions on $\PP^n$; 
\item[(2)] isomorphism classes of local $(n+1)$-dimensional algebras.
\end{enumerate}
\end{proposition}

Moreover, there is a bijection between equivalence classes of cyclic effective $\GG_a^n$-actions on $\PP^m$ and classes of equivalence of faithful cyclic rational $(m+1)$-dimensional representations. Recall that a linear representation $\rho\colon G\rightarrow \mathrm{GL}(V)$ is called {\it cyclic} if there exists a vector $v\in V$ such that the orbit $Gv$ spans $V$. We call $v$ the {\it cyclic vector} of $\rho$.

\begin{proposition}\label{hc}
{\rm \cite[Theorem 2.14]{HTsch}} There is a one-to-one correspondence between

\begin{enumerate}
\item[(1)] equivalence classes of faithful cyclic rational representations $$\rho: \mathbb{G}_{a}^{n}  \rightarrow \mathrm{GL}_{s}(\mathbb{K});$$ 
\item[(2)] isomorphism classes of pairs (A,U), where A is a local $s$-dimensional
algebra with the maximal ideal $\mathfrak{m}$ and U is an $n$-dimensional subspace of $\mathfrak{m}$
that generates A as an algebra with unit.
\end{enumerate}
\end{proposition}

The generalization of Proposition \ref{p1} was obtained independently by Knop and Lange in \cite{KL}.

\begin{proposition}\label{kl}
{\rm \cite[Corollary 5.2]{KL}} There is a one-to-one correspondence
between

\begin{enumerate}
\item[(1)] equivalence classes of commutative actions of linear algebraic group $ \GG_a^l\times\GG_m^r$ on~$\PP^n$; 
\item[(2)] isomorphism classes of $(n+1)$-dimensional commutative
$\mathbb{K}$-algebras with exactly $r+1$ maximal ideals.
\end{enumerate}

\end{proposition}

Commutative actions on projective subvarieties $X \subseteq \PP^m$ induced by an action $G \times \PP^m \to \PP^m$ can be described in terms of $(m+1)$-dimensional commutative algebras equipped with some additional data. This approach was used by Sharoiko in \cite{Sh} to classify additive actions on smooth projective quadrics.

\begin{proposition}\label{shh}
{\rm \cite[Theorem 4]{Sh}} A smooth quadric $Q_n \subseteq \PP^{n+1}$ admits a unique additive action up to equivalence.
\end{proposition}

\section{A generalization of Hassett-Tschinkel correspondence}

From now on let $G \simeq \mathbb{G}_{a}^{l} \times \mathbb{G}_{m}^{r}$ be a commutative linear algebraic group and~$A$ be a finite-dimensional commutative associative algebra with unit. Let $R$ be its radical, i.e. $R = \{a \in A \,| \ \exists\, n \in \mathbb{N}: a^{n}=0\}$.

\begin{definition}
Let $\rho$ and $\rho'$ be two representations of $G$, $V$ and $V'$ be corresponding spaces of representations. Then $\rho$ and $\rho'$ are called {\it equivalent} if there exists an isomorphism $\phi: V \to V'$ such that $\rho'(g) \circ \phi=\phi \circ \rho(g) \ \ \forall g \in G.$
\end{definition}

The following theorem is a generalization of Proposition~\ref{hc} to the case of commutative but not necessary additive group.

\begin{theorem} \label{t1} 
There is a bijection between

\begin{enumerate}
\item[(1)] equivalence classes of faithful cyclic rational representations $$\rho: \mathbb{G}_{a}^{l} \times \mathbb{G}_{m}^{r} \rightarrow \mathrm{GL}_s(\mathbb{K});$$ 
\item[(2)] isomorphism classes of pairs (A,U), where A is an $s$-dimensional
algebra and U is an $(l+r)$-dimensional subspace of A
that generates A as an algebra with unit such that $\mathrm{dim}(U \cap R) = l$.
\end{enumerate}

\end{theorem}

\begin{proof}

Let  $\rho: \mathbb{G}_{a}^{l} \times \mathbb{G}_{m}^{r} \rightarrow \mathrm{GL}_s(\mathbb{K})$ be a faithful cyclic rational representation. The differential
defines a representation $d\rho: \mathfrak{g} \rightarrow \mathfrak{gl}_s(\mathbb{K})$ of the tangent algebra $\mathfrak{g}=\mathrm{Lie}(G).$  Let us denote 
$U=d\rho(\mathfrak{g})$. Consider the subalgebra with unit $$A\subseteq \mathrm{Mat}_{s\times s}(\mathbb{K})$$ generated by $U$.
Let us prove that the algebra $A$ is isomorphic as a vector space to~$\mathbb{K}^{s}.$
Indeed, let us fix a cyclic vector $v$ for $\rho$ and consider the map 
$$\xi: A \rightarrow \mathbb{K}^{s}, \ \ a \rightarrow a \cdot v.$$ 
Every element of  $\rho(G)$ is the exponent of an element of $U$. But exponent of a matrix can be realized as a polynomial in this matrix. So $ \rho(G) \subseteq A$. Then $$\mathbb{K}^{s}=\langle\xi(G)\rangle\subseteq \xi(A)\Rightarrow \xi(A)=\mathbb{K}^{s}.$$ 
Therefore, $\xi$ is a surjection and $\mathrm{dim}(A)\geq s.$ Moreover, we have implications
$$a \in \mathrm{Ker}(\xi) \Rightarrow a \cdot v=0 \Rightarrow Aav=0 \Rightarrow aAv=0 \Rightarrow a\mathbb{K}^{s}=0 \Rightarrow a=0.$$ 
Thus $\xi$ is an isomorphism and $\mathrm{dim}(A)=s.$

The subspace $U$ generates $A$ as an algebra with unit. Since $G=\mathbb{G}_{a}^{l} \times \mathbb{G}_{m}^{r}$, we obtain
$\mathfrak{g}=\mathfrak{g}_{a} \oplus \mathfrak{g}_{m}$ and $d\rho(\mathfrak{g})=d\rho(\mathfrak{g}_{a}) \oplus d\rho(\mathfrak{g}_{m})$, where $d\rho(\mathfrak{g}_{a})$ and $d\rho(\mathfrak{g}_{m})$ are nilpotent and diagonalizable matrices, respectively. Therefore, $\mathrm{dim}(U)=l+r=n$ and 
$$\dim(U \cap R) = l.$$
Conversely, let $A$ be an $s$-dimensional algebra and $U$ be a subspace that generates~$A$ as an algebra with unit. The left action of~$A$ on itself induces an embedding of $A$ into the space of operators on $A$
$$\tau: A \subseteq L(A).$$
Since  $U \subseteq A$, $\mathrm{dim}(U)=l+r=n$ and $\mathrm{dim}(U \cap R)=l$, the group $G =~\{\mathrm{exp}(a) | a \in~U \}$ is isomorphic to $\mathbb{G}_{a}^{l} \times \mathbb{G}_{m}^{r}$ and we obtain a representation
$$\rho: G \rightarrow \mathrm{GL}(A).$$
The subspace $\langle G 
\cdot 1 \rangle$ is $G$-invariant, so it is $U$-invariant, and since $U$ generates $A$ as an algebra, it is $A$-invariant as well. As a result, we obtain $A=A \cdot 1 \subseteq \langle G \cdot 1 \rangle$. Moreover, if $g \in \mathrm{Ker}(\rho)$ then $g \cdot 1=1$, thus $g=1$ and we get that representation is cyclic and faithful.

One can check that these correspondences between representations and algebras are inverse to each other and two algebras with generating subspaces are isomorphic if and only if corresponding representations are equivalent.
\end{proof}

\begin{remark}\label{rr}
The algebra $A\subseteq\mathrm{Mat}_{s\times s}(\KK)$ can be decomposed into a direct sum of $R$ and $A_d$, where $A_d$ is the subalgebra of diagonalizable elements. We have
$$
U=U_a\oplus U_m, \text{where}\  U_a=R\cap U=\mathrm{d}\rho(\mathfrak{g}_a), \ U_m=A_d\cap U=\mathrm{d}\rho(\mathfrak{g}_m).
$$
\end{remark}

As we see in the proof of the previous theorem, if we fix a cyclic vector $v$ in $\KK^s$, we obtain an isomorphism $\xi\colon A\rightarrow \KK^s$, $\xi(1)=v$. If we have a scalar product $(\cdot,\cdot)$ on $\KK^s$, the isomorphism $\xi$ induces a scalar product $(\cdot,\cdot)_v$ on $A$ via 
$$(a,a')_v=(\xi(a),\xi(a')).$$ 

\begin{lemma}\label{cc}
Under the assumptions of Theorem~\ref{t1}, if the representation $\rho$ is orthogonal with respect to some scalar product $(\cdot,\cdot)$, then for induced scalar product~$(\cdot,\cdot)_v$ on~$A$ we have 
$$
(ua_{1},a_{2})_v+(a_{1},ua_{2})_v=0 \ \ \forall u \in U\ \  \forall a_{1}, a_{2} \in A,
$$
and
$$
(1,1)_v=(v,v).
$$
\end{lemma}
\begin{proof}
The isomorphism $\xi\colon A\rightarrow \KK^s$  is $\rho(G)$-equivariant. Indeed, 
$$\xi(\rho(g)a)=\rho(g)a\cdot v=\rho(g)\cdot\xi(a).$$ The representation $\rho$ is orthogonal, this means 
$$(v_1,v_2)=(\rho(g)\cdot v_1, \rho(g)\cdot v_2)$$
for all $v_1, v_2\in \KK^s$, $g\in G$.
Since $U=\mathrm{d}\rho(\mathfrak{g})$, every $u\in U$ can be obtained as the tangent vector to some curve $\rho(g(t))$ at the point $\rho(g(0))=E$. Applying differential to the previous equality gives 
$$
(ua_{1},a_{2})_v+(a_{1},ua_{2})_v=0 \ \ \forall u \in U\ \  \forall a_{1}, a_{2} \in A.
$$

Since $\xi(1)=v$, we have $(1,1)_v=(\xi(1),\xi(1))=(v,v)=0$.
\end{proof}

\section{Main results}

Let $Q_n\subseteq \PP^{n+1}$ be a smooth quadric. We can assume that $Q_n$ is given by the equation $x_0x_{n+1}=x_1^2+\ldots+x_n^2$, where $[x_0:\ldots :x_{n+1}]$ are homogeneous coordinates on $\PP^{n+1}$. It is easy to see that the following formula gives an additive action on  $Q_n$, see \cite{Sh}: 
\begin{multline}\label{act}
(s_1,\ldots,s_n)\cdot[x_0:\ldots :x_{n+1}]=\\
=\left[x_0+\sum_{i=1}^n(2x_is_i+x_{n+1}s_i^2):x_1+s_1x_{n+1}:\ldots:x_n+s_nx_{n+1}:x_{n+1}\right],
\end{multline}
where $(s_1,\ldots,s_n) \in \GG_a^n.$

When $n=2$ we can define a $\GG_a\times\GG_m$-action on $Q_2$ with an open orbit. To describe it we note that $Q_2$ can be considered as the set of degenerate $2\times 2$-matrices:
$Q_2=\left\{
\begin{pmatrix}
x_0& x_1\\
x_2& x_3
\end{pmatrix}
\mid 
x_0x_3=x_1x_2
\right\}.$
Consider the following $\GG_a\times\GG_m$-action on $Q_2$:
\begin{equation}\label{pc}
(s,t)\cdot \begin{pmatrix}
x_0& x_1\\
x_2& x_3
\end{pmatrix}=
\begin{pmatrix}
1& s\\
0& 1
\end{pmatrix}
\begin{pmatrix}
x_0& x_1\\
x_2& x_3
\end{pmatrix}
\begin{pmatrix}
t& 0\\
0& 1
\end{pmatrix},\ (s,t) \in \GG_a\times\GG_m.
\end{equation}

Also there are $\GG_m^n$-actions on $Q_n$ with an open orbit for $n=1, 2$ given by
\begin{equation}\label{odnom}
t\cdot (x_0:x_1:x_2)=(tx_0:t^{-1}x_1:x_2), \ \text{where $Q_1$ is given by } x_0x_1=x_2^2,\ t \in \GG_m;
\end{equation}
\begin{multline}\label{dvum}
(t_1,t_2)\cdot (x_0: x_1: x_2: x_3)=(t_1x_0: t_1^{-1}x_1:t_2x_2: t_2^{-1}x_3),\\ \text{where $Q_2$ is given by } x_0x_1=x_2x_3,\ (t_1,t_2) \in \GG_m^2.
\end{multline}

The main result of this paper is the following theorem.
\begin{theorem}\label{maint}
Let $Q_n\subseteq \PP^{n+1}$ be a smooth quadric. 

\begin{enumerate}
\item[(i)] If $n=1$, then there are two different commutative actions  on  $Q_1$ up to equivalence. The first one is the additive action given by (\ref{act}). The second one is the multiplicative action given by (\ref{odnom}).

\item[(ii)] If $n=2$, then there are three different commutative actions  on  $Q_2$ up to equivalence. The first one is the additive action given by (\ref{act}), the second one is the $\GG_a\times\GG_m$-action given by (\ref{pc}) and the third one is the multiplicative action given by~(\ref{dvum}).

\item[(iii)] If $n\geq 3$, then the action given by (\ref{act}) is the unique commutative action on~$Q_n$ up to equivalence.
\end{enumerate}
\end{theorem}

To prove this theorem we need some preparations. Assume that $Q_{n}$ is defined by a homogeneous equation
$$f(x_0,x_1,...,x_{n+1})=0.$$
Let $F$ be the polarization of the polynomial $f$, i.e. $F$ is a bilinear form on $\mathbb{K}^{n+2}$ such that $f(v)=F(v,v)$. Since $Q_{n}$ is smooth, $F$ is non-degenerate.

\vspace{5mm}
The following lemma is well known, see, e.g.~\cite[Lemma~2]{Sh}.
\begin{lemma} \label{l1} 
Let $Q_n\subseteq \mathbb{P}^{n+1}$ be a smooth quadric. Then the group of regular automorphisms $\mathrm{Aut}(Q_{n})$ coincides with $\mathrm{PSO}_{n+2}(\mathbb{K})$.
\end{lemma}
Now let $G\cong \GG_a^l\times \GG_m^r$ act effectively on $Q_{n}$ with an open orbit. This lemma implies that every $G$-action on $Q_{n}$ can be lifted to a $G$-action on~$\mathbb{P}^{n+1}$. Two actions on $Q_{n}$  are equivalent if and only if their liftings are  equivalent. Denote by $\pi\colon \mathrm{SO}_{n+2}(\mathbb{K})\rightarrow \mathrm{PSO}_{n+2}(\mathbb{K})$ the natural homomorphism. If $n$ is odd, $\pi$ is an isomorphism. If $n$ is even, we have $\mathrm{Ker}\,\pi\cong\mathbb{Z}/2\mathbb{Z}$. 

For every $G\subseteq \mathrm{PSO}_{n+2}(\mathbb{K})$ we can consider the preimage  $\widehat{G}=~\pi^{-1}(G)~\subseteq~\mathrm{SO}_{n+2}(\mathbb{K})$.  So an effective $G$-action on $Q_{n}$ induces an orthogonal linear representation with respect to scalar product $F$
$$\rho\colon\widehat{G}\rightarrow \mathrm{SO}_{n+2}(\mathbb{K}).$$

Since the kernel of $\widehat{G}\rightarrow G$ is finite, the neutral component $\widehat{G}^0$ is isomorphic to $\GG_a^l\times \GG_m^r$. Since $G=\pi(\widehat{G})$-action on $Q_n$ has an open orbit,  $\pi(\widehat{G}^0)$-action on $Q_n$ has an open orbit as well.

\begin{lemma}
Let a linear algebraic group $G$ act effectively on $Q_{n}$ with an open orbit. Then the induced representation $\rho\colon\widehat{G}^0\rightarrow \mathrm{SO}_{n+2}(\mathbb{K})$ is cyclic.
\end{lemma}
\begin{proof}
Let $x\in Q_{n}\subseteq \mathbb{P}^{n+1}(\KK)$ be an element of the open $\pi(\widehat{G}^0)$-orbit. Let $x=\langle v \rangle$ for $v\in\KK^{n+2}$. Then $\rho(\widehat{G}^0)v$ is not contained in any proper subspace $W\subseteq \KK^{n+2}$ since~$Q_n$ is not contained in any proper subspace of $ \mathbb{P}^{n+1}(\KK)$. So $v$ is a cyclic vector for~$\rho$.
\end{proof}
Since $Q_n$ is given by the equation $f(x)=0$, the cyclic vector $v$  is isotropic, i.e. $F(v,v)=0$.

By definition $\widehat{G}\hookrightarrow \mathrm{SO}_{n+2}(\mathbb{K})$ is an injection, thus the representation $\rho\colon\widehat{G}^0~\rightarrow~\mathrm{SO}_{n+2}(\mathbb{K})$ is faithful. So we can apply Theorem~\ref{t1} and Lemma~\ref{cc}. We obtain an $(n+2)$-dimensional algebra $A$ with $n$-dimensional generating  subspace~$U$ such that $\dim U\cap R=l$. Also we obtain scalar product $F(\cdot,\cdot)_v$ on $A$ such that $F(1,1)_v=0$ and

\begin{equation}\label{ff}
F(ua_{1},a_{2})_v+F(a_{1},ua_{2})_v=0 \ \ \forall u \in U\ \  \forall a_{1}, a_{2} \in A.
\end{equation}

The following lemma is analogous to a result obtained in \cite[Lemma~6]{Sh}, see also \cite[Lemma~2]{AP}.

\begin{lemma} \label{l3}
The scalar product $F(\cdot,\cdot)_v$ has the following properties:

\smallskip

\begin{enumerate}
\item[(i)] every $u \in U$ is orthogonal to $1$; 
\item[(ii)] for all $u_1, u_2 \in U$ we have $u_1u_2 \in U^{\bot}$ ; 
\item[(iii)] the restriction of the form $F$ to $U$ is non-degenerate;
\item[(iv)] the dimension of $(U^{\bot})$ equals $2$;
\item[(v)] if $l\neq 0$, then $\mathrm{dim}(U^{\bot} \cap R) = 1$.
\end{enumerate}

\end{lemma}

\begin{proof}
Assertion (i) follows from (\ref{ff}) with $a_1=a_2=1$. 

(ii) Note that
$$F(u_1u_2,u_3)_v=-F(u_1,u_2u_3)_v=F(1,u_1u_2u_3)_v=-F(u_3,u_1u_2)_v \ \ \forall u_1,u_2,u_3 \in U.$$
Therefore $F(u_1u_2,u_3)_v=0$ and $u_1u_2 \in U^{\bot}\ \ \forall u_1, u_2 \in U$.

(iii) Assume that for some $0 \neq u \in U$ we have  $U\subseteq \langle u\rangle^\bot$. Hence, $\langle u\rangle^\bot=\langle 1, U\rangle$. But $\langle 1\rangle^\bot=\langle 1, U\rangle$. Since $1\notin U$, this gives a contradiction.

(iv) We have $\mathrm{dim}(A)=n+2$, $\mathrm{dim}(U)=n$ and $F$ is non-degenerate. Therefore, $\dim U^{\bot} = 2$. 

(v) Let $b\in U\cap R$. By (iii) there is $u\in U$ such that $F(b,u)_v=-F(bu,1)_v \neq 0$. By (ii) we have $bu \in U^{\bot} \cap R$.
Hence, $\mathrm{dim}(U^{\bot} \cap R) \geq 1$. Since $\mathrm{dim}(U^{\bot})=2$ and $1 \in U^{\bot} \setminus R$, we obtain $\mathrm{dim}(U^{\bot} \cap R) = 1$.
\end{proof}

\begin{proposition}\label{propp}
Let $l\geq 1$, $r\geq 1$ and $n=l+r\geq 3$. Then there is no effective $\GG_a^l\times \GG_m^r$-action on $Q_n$ with an open orbit. 
\end{proposition}
\begin{proof}
Suppose there exists such an action. Let us fix $u\neq 0 \in U_m$, see Remark~\ref{rr}. Then $u^2\in U^\bot$ by Lemma~\ref{l3}(ii). But $u^2\in A_s$. By  Lemma~\ref{l3}(v) we obtain $\dim(U^{\bot}~\cap~A_s)\leq~1$. Lemma~\ref{l3}(i) implies $U^{\bot} \cap A_s=\langle 1\rangle$. Hence, $u^2=\lambda 1$, where $\lambda\neq 0\in \KK$. Let us consider the multiplication operator by $u$: 
$$\phi_u: A \rightarrow A, \qquad a\mapsto ua.$$
Since $u^2=1$, the operator $\phi_u$ is non-degenerate. But it follows from Lemma~\ref{l3}(ii) that $\phi_u(U)\subseteq U^\bot$, where $\dim U=n\geq 3$ and $\dim U^\bot=2$ by Lemma~\ref{l3}(iv). 

A contradiction.
\end{proof}

\begin{lemma}\label{tm}
\begin{enumerate}
\item[(i)]  Let $n=r > 2$. Then there is no effective $\GG_m^r$-action on~$Q_n$ with an open orbit.
\item[(ii)] If there exists an effective $\GG_m^n$-action on $Q_n$ with an open orbit, then it is unique up to equivalence. 
\end{enumerate}
\end{lemma}
\begin{proof}
(i) It is well known that the dimension of the maximal torus in $\mathrm{SO}_{n+2}(\mathbb{K})$ is equal to  $\frac{n+2}{2}$ if $n$ is even, and it is equal to $\frac{n+1}{2}$ if $n$ is odd. Since $n>2$, the dimension of the maximal torus in $\mathrm{PSO}_{n+2}(\mathbb{K})$ is less than $n$.

(ii) All maximal tori in $\mathrm{PSO}_{n+2}(\mathbb{K})$ are conjugated. Therefore, all effective $\GG_m^n$-action on $Q_n$ are conjugated by an automorphism of $Q_n$.
\end{proof}

\begin{lemma}\label{un}
There exists a unique up to equivalence $\GG_a\times\GG_m$-action on $Q_2$ with an open orbit.
\end{lemma}

\begin{proof}
A $\GG_a\times\GG_m$-action on $Q_2$ with an open orbit corresponds to a 4-dimensional algebra $A$ with 2-dimensional generating subspace $U=U_a\oplus U_m$, where $$\dim U_a=\dim U_m=1.$$ Fix $u\neq 0\in U_a$, $w\neq 0\in U_m$. Since $w$ is a semisimple element of $A\subseteq \mathrm{Mat}_{4\times 4}(\KK)$, $w^2$ is a semisimple element as well. Therefore, $w^2\in\langle 1,w\rangle$. By Lemma~\ref{l3}(i), (ii) and (iii) $1\in U^\bot$, $w^2\in U^\bot$ and $w\notin U^\bot$. Therefore, $w^2=\lambda 1$, where $\lambda\in \KK\setminus\{0\}$. We can change $w$ to get $w^2=1$.

We obtain $$F(w,w)_v=-F(1,w^2)_v=F(1,1)_v=0.$$ By Lemma~\ref{l3}(iii) the restriction of $F$ to $U$ is non-degenerate. Hence, $F(u,w)_v\neq 0$. We can change $u$ to get $F(u,w)_v=-1$. 

Denote $h=uw$. Then $h\in U^\bot\cap R$, $hw=uw^2=u$ and $$F(1,h)_v=-F(u,w)_v=1.$$ In particular, $h\neq 0$. So $\{1,u,w,h\}$ is a basis of $A$.

By Lemma~\ref{l3}(ii) the element $u^2$ belongs to $U^\bot$. Therefore, $u^2\in U^\bot\cap R=\langle h \rangle$, i.e. $u^2=\lambda h$. Hence, $u^2w=\lambda hw=\lambda uw^2=\lambda u$. Since $u$ is a nilpotent element we have $\lambda=0$, i.e. $u^2=0$. This implies $uh=u^2w=0$ and $h^2=u^2w^2=0$. Therefore, $F(u,u)_v=-F(u^2,1)_v=0$ and $F(h,h)_v=-F(h^2,1)_v=0$.

So, we obtain the following tables for multiplication and scalar product in~$A$:

\begin{center}
    \begin{tabular}{c||c|c|c|c}
        $\cdot$&1&u&w&h\\
        \hline
          \hline
         1&1&u&w&h\\
         \hline
         u&u&0&h&0\\
         \hline
         w&w&h&1&u\\
         \hline
         h&h&0&u&0
    \end{tabular}
    \qquad \qquad \qquad
        \begin{tabular}{c||c|c|c|c}
        F$(\cdot,\cdot)_v$&1&u&w&h\\
        \hline
          \hline
         1&0&0&0&1\\
         \hline
         u&0&0&-1&0\\
         \hline
         w&0&-1&0&0\\
         \hline
         h&1&0&0&0
    \end{tabular}
\end{center}

This implies uniqueness of $A$ up to isomorphism and hence uniqueness of the corresponding action up to equivalence.

\end{proof}
\begin{proof}[Proof of Theorem~\ref{maint}]
Our goal is to describe all commutative actions of $\GG_a^l\times\GG_m^r$ on $Q_n$. Proposition~\ref{propp} implies that the only possible actions are additive actions ($r=0$), multiplicative actions ($l=0$) and $\GG_a\times\GG_m$-actions. By Proposition~\ref{shh} for every $n$ there is a unique additive action on~$Q_n$ given by (\ref{act}). By Lemma~\ref{un} there exists a unique effective $\GG_a\times\GG_m$-action on~$Q_2$ with an open orbit. It is easy to check that (\ref{pc}) gives such an action. Formulas (\ref{odnom}) and (\ref{dvum}) give effective $\GG_m^n$-actions on~$Q_n$ with open orbits for $n=1$ and $2$. By Lemma~\ref{tm} these two actions are the only effective $\GG_m^n$-actions on~$Q_n$ with open orbits.
\end{proof}


\begin{thebibliography}{99}
\bibitem{A}
Ivan Arzhantsev, {\it Flag varieties as equivariant compactifications of $\GG_a^n$.} Proc. Amer. Math. Soc. {\bf 139} (2011), no.~3, 783-786.

\bibitem{AP} 
Ivan Arzhantsev and  Andrey Popovskiy, 
\textit{Additive Actions on Projective Hypersurfaces.} Automorphisms in birational and affine geometry, Springer Proc. Math. Stat., vol. 79, Springer, Cham, 2014, 17–33. 

\bibitem{AR}
Ivan Arzhantsev  and Elena Romaskevich, {\it Additive actions on toric varieties.} Proc. Amer. Math. Soc. {\bf 145} (2017), no.~5, 1865-1879.

\bibitem{ASh}
Ivan Arzhantsev and Elena Sharoiko, {\it Hassett-Tschinkel correspondence: Modality and projective
hypersurfaces.} J. Algebra {\bf 348} (2011), no.~1, 217-232.

\bibitem{B}
Ivan Bazhov, {\it Additive Structures on Cubic Hypersurfaces}. (2013) 8 pp. [arXiv:1307.6085].

\bibitem{D}
Rostislav Devyatov, {\it Unipotent commutative group actions on flag varieties and nilpotent
multiplications}. Transform. Groups {\bf 20} (2015), no.~1, 21-64.

\bibitem{F} 
Evgeny Feigin, {\it $\GG_a^M$ degeneration of flag varieties}. Selecta Math. (N.S.) {\bf 18} (2012), no.~3, 513-537.

\bibitem{FH}
Baohua Fu and Jun-Muk Hwang, {\it Uniqueness of equivariant compactifications of $\mathbb{C}^n$ by a Fano manifold of Picard number 1}. Math. Res. Lett. {\bf 21} (2014), no.~1, 121–125.

\bibitem{HTsch} 
Brendan Hassett and Yuri Tschinkel, \textit{Geometry of equivariant compactifications of $\mathbb{G}_{a}^{n}$}. Int. Math. Res. Notices \textbf{20} (1999), 1211-1230.


\bibitem{KL} 
Friedrich Knop and Herbert  Lange,
\textit{Commutative algebraic groups and intersections of quadrics}. Math. Ann. {\bf 267} (1984), no.~4, 555-571.

\bibitem{L}
Konstantin Loginov, {\it Hilbert-Samuel sequences of homogeneous finite type}. J. Pure Appl. Algebra {\bf 221} (2017), 821-831.
 
\bibitem{Sh} 
Elena Sharoiko, 
\textit{Hassett-Tschinkel correspondence and automorphisms of a quadric.}
Sb. Math. \textbf{200} (2009), no.~11, 1715-1729.

\end{thebibliography}
\end{document}